\documentclass[12pt]{amsart}
\usepackage[headings]{fullpage}
\usepackage{amssymb,epic,eepic,epsfig,amsbsy,amsmath,amscd,graphicx}
\usepackage[all]{xy}

\numberwithin{equation}{section}
                        \theoremstyle{plain}
\usepackage{mathrsfs}

\newtheorem{theorem}{Theorem}[section]
\newtheorem{thm}{Theorem}
\newtheorem{lemma}[theorem]{Lemma}

\newtheorem{proposition}[theorem]{Proposition}
\newtheorem{conjecture}{Conjecture}

\theoremstyle{definition}
\newtheorem{remark}[theorem]{Remark}

\newcommand\no[1]{}

\def\BC{\mathbb C}

\def\BZ{\mathbb Z}

\def\CA{\mathcal A}

\def\CR{\mathcal R}

\def\CT{\mathcal T}

\def\fp{\mathfrak p}
\def\ft{\mathfrak t}

\def\ve{\varepsilon}
\def\be { \begin{equation} }
\def\ee { \end{equation} }

\begin{document}

\title[Strong AJ conjecture]{The strong AJ conjecture for cables of torus knots}

\author[Anh T. Tran]{Anh T. Tran}
\address{Department of Mathematics, The Ohio State University, Columbus, OH 43210, USA}
\email{tran.350@osu.edu}

\begin{abstract}
The AJ conjecture, formulated by Garoufalidis, relates the A-polynomial and the colored Jones polynomial of a knot in the 3-sphere. It has been confirmed for all torus knots, some classes of two-bridge knots and pretzel knots, and most cabled knots over torus knots. The strong AJ conjecture, formulated by Sikora, relates the A-ideal and the colored Jones polynomial of a knot. It was confirmed for all torus knots. In this paper we confirm the strong AJ conjecture for most cabled knots over torus knots.
\end{abstract}

\thanks{2010 {\em Mathematics Classification:} Primary 57N10. Secondary 57M25.}

\maketitle

\setcounter{section}{-1}

\section{Introduction}

\subsection{The colored Jones polynomial} For a knot $K$ in $S^3$ and a positive integer $n$, let $J_K(n) \in \BZ[t^{\pm 1}]$ denote the $n$-colored Jones polynomial of $K$ with zero-framing. The polynomial $J_K(n)$ is the quantum link invariant, as defined by Reshetikhin and Turaev \cite{RT}, associated to the Lie algebra $sl_2(\BC)$, with the color $n$ standing for the irreducible $sl_2(\BC)$-module $V_{n}$ of dimension $n$. Here we use the functorial normalization, i.e. the one for which the colored Jones polynomial of the unknot $U$ is
$$J_U(n)=[n] := \frac{t^{2n}- t^{-2n}}{t^2 -t^{-2}}.$$ 

It is known that $J_K(1)=1$ and $J_K(2)$ is the usual Jones polynomial \cite{Jones}. The colored Jones polynomials of higher colors are more or less the usual Jones polynomials of parallels of the knot. The color $n$ can be assumed to take negative integer values by setting $J_K(-n) := - J_K(n)$. In particular, one has $J_K(0)=0$.

\subsection{The recurrence ideal and recurrence polynomial} Consider a discrete function $f: \BZ \to \CR:=\BC[t^{\pm 1}]$ and define the linear operators $L, M$ acting on such functions by
$$(Lf)(n) := f(n+1), \qquad (Mf )(n) := t^{2n} f(n).$$
It is easy to see that $LM = t^2 ML$, and that $L^{\pm 1}, M^{\pm 1}$ generate the quantum torus $\CT$, a non-commutative ring with presentation
$$ \mathcal T := \mathbb \CR\langle L^{\pm1}, M^{\pm 1} \rangle/ (LM - t^2 ML).$$

Let $$\mathcal A_K := \{ P \in \mathcal T \mid  P J_K=0\},$$ which is a left-ideal of $\mathcal T$, called the {\em recurrence ideal} of $K$. It was proved in \cite{GL} that for every knot $K$, the
recurrence ideal $\mathcal A_K$ is non-zero. An element in $\mathcal A_K$ is called a recurrence relation for the colored Jones polynomial of $K$.

The ring $\mathcal T$ is not a principal left-ideal domain, i.e. not every left-ideal of $\mathcal T$ is generated by one element. By adding all  inverses of polynomials in $t,M$ to
$\mathcal T$ one gets a principal left-ideal domain $\tilde\CT$, c.f. \cite{Ga04}. The ring $\tilde{\CT}$ can be formally defined as follows. Let
$\CR(M)$ be the fractional field of the polynomial ring $\CR[M]$.
Let $\tilde \CT$ be the set of all Laurent polynomials in the
variable $L$ with coefficients in $\CR(M)$:
$$\tilde
\CT :=\{\sum_{j\in \BZ}a_j(M) L^j \,\, | \quad a_j(M)\in \CR(M),
\,\,\, f_j=0  \quad \text{almost everywhere} \},
$$
and define the product in  $\tilde \CT$ by $a(M) L^{k} \cdot b(M)
L^{l} :=a(M)\, b(t^{2k}M) L^{k+l}.$

The left ideal extension $\tilde \CA_K :=\tilde \CT \CA_K$ of $\CA_K$ in
$\tilde \CT$ is then generated  by a polynomial
$$\alpha_K(t;M,L) = \sum_{j=0}^{d} \alpha_{K,j}(t,M) \, L^j,$$
where $d$ is assumed to be minimal and all the
coefficients $\alpha_{K,j}(t,M)\in \BZ[t^{\pm1},M]$ are assumed to
be co-prime. That $\alpha_K$ can be chosen to have integer
coefficients follows from the fact that $J_K(n) \in
\BZ[t^{\pm1}]$. The polynomial $\alpha_K$ is defined up to a polynomial in $\mathbb Z[t^{\pm 1},M]$. We call $\alpha_K$ the {\em recurrence polynomial} of $K$.

\subsection{The AJ conjecture} The colored Jones polynomials are powerful invariants of knots, but little is known about their relationship with classical invariants like the fundamental group. Inspired by the theory of noncommutative A-ideals of Frohman, Gelca and Lofaro \cite{FGL, Ge} and the theory of $q$-holonomicity of quantum invariants of Garoufalidis and Le \cite{GL}, Garoufalidis \cite{Ga04} formulated the following conjecture that relates the A-polynomial and the colored Jones polynomial of a knot in the 3-pshere.

\begin{conjecture}{\bf (AJ conjecture)} For every knot $K$ in $S^3$, $\alpha_K \mid_{t=-1}$ is equal to the $A$-polynomial, up to a polynomial depending on $M$ only.
\end{conjecture}

The $A$-polynomial of a knot was introduced by Cooper et al. \cite{CCGLS}; it describes the $SL_2(\BC)$-character variety of the knot complement as viewed from the boundary torus. Here in the definition of the $A$-polynomial, we also allow the factor $L-1$ coming from the abelian component of the character variety of the knot group. Hence the $A$-polynomial in this paper is equal to $L-1$ times the $A$-polynomial defined in \cite{CCGLS}.

The AJ conjecture has been confirmed for the trefoil knot and the figure eight knot (by Garoufalidis \cite{Ga04}), all torus knots (by Hikami \cite{Hi}, Tran \cite{Tr}), some classes of two-bridge knots and pretzel knots including most double twist knots and $(-2,3,6n \pm 1)$-pretzel knots (by Le \cite{Le}, Le and Tran \cite{LT}), the knot $7_4$  (by Garoufalidis and Koutschan \cite{GK}), and most cabled knots over torus knots (by Ruppe and Zhang \cite{RZ}). 

\subsection{Main result} For a finitely generated group $G$, let $\chi(G)$ denote the $SL_2(\BC)$-character variety of $G$, see \cite{CS}. For a manifold $Y$ we use $\chi(Y)$ also to denote $\chi(\pi_1(Y))$. Suppose  $G=\BZ^2$.
Every pair  of generators $\mu, \lambda$ will define an isomorphism
between $\chi(G)$ and $(\BC^*)^2/\tau$, where $(\BC^*)^2$ is the
set of non-zero complex pairs $(M,L)$ and $\tau$ is the involution
$\tau(M,L):=(M^{-1},L^{-1})$. For an algebraic set $V$ (over $\BC$), let $\BC[V]$ denote the ring
of regular functions on $V$.  For example, $\BC[(\BC^*)^2/\tau]=
\ft^\sigma$, the $\sigma$-invariant subspace of  $\ft:=\BC[M^{\pm
1},L^{\pm 1}]$, where $\sigma(M^kL^l):= M^{-k}L^{-l}.$

Let $K$ be a knot in $S^3$ and $X=S^3 \setminus K$ its complement. 
The boundary of
$X$ is a torus whose fundamental group  is free abelian of rank
two. An orientation of $K$ will define a unique pair of an
oriented meridian $\mu$ and an oriented longitude $\lambda$ such that the linking
number between the longitude and the knot is zero. The pair provides
an identification of $\chi(\partial X)$ and $(\BC^*)^2/\tau$
which actually does not depend on the orientation of $K$.

The inclusion $\partial X \hookrightarrow X$ induces an algebra homomorphism
$$\theta: \BC[\chi(\partial X)]  \equiv \ft^\sigma \longrightarrow
\BC[\chi(X)].$$
We call the kernel $\fp$ of $\theta$ the {\em $A$-ideal} of $K$; it is an ideal of $\ft^\sigma$. The $A$-ideal was first introduced in \cite{FGL}; it determines the $A$-polynomial of $K$. In fact $\fp=(A_K \cdot \ft)^{\sigma}$, the $\sigma$-invariant part of the ideal $A_K \cdot \ft \subset \ft$ generated by the $A$-polynomial $A_K$.

The involution $\sigma$ acts on the quantum torus $\CT$ also by $\sigma(M^kL^l)= M^{-k}L^{-l}$. Let $\CA_K^\sigma$ be the $\sigma$-invariant part of the recurrence ideal $\CA_K$; it is an ideal of $\CT^{\sigma}$. 

Sikora \cite{Si} formulated the following conjecture that relates the A-ideal and the colored Jones polynomial of a knot in the 3-pshere.

\begin{conjecture} {\bf (Strong AJ conjecture)}
Suppose $K$ is a knot in $S^3$. Then $$\sqrt{\CA_K^{\sigma} \mid_{t=-1}}=\fp.$$ 
\end{conjecture}

Here $\sqrt{\CA_K^{\sigma} \mid_{t=-1}}$ denotes the radical of the ideal $\CA_K^{\sigma} \mid_{t=-1}$ in the ring $\ft^{\sigma}$. 

The strong AJ conjecture was confirmed for the trefoil knot (by Sikora \cite{Si}) and all torus knots (by Tran \cite{Tr}). In this paper we consider the strong AJ conjecture for cabled knots over torus knots. Recall that the set of non-trivial torus knots $T(p,q)$ in $S^3$ can be indexed by pairs of relatively prime integers $(p,q)$ satisfying $|p|>q \ge 2$. Also recall that an $(r,s)$-cabled knot over a knot $K$ in $S^3$ is the knot which can be embedded in the boundary torus of a tubular neighborhood of $K$ in $S^3$ as a curve of slope $r/s$ with respect to the meridian/longitude coordinates of $K$ satisfying $(r,s)=1$ and $s \ge 2$.

Our main result is the following.

\begin{thm}
\label{th2}
The strong AJ conjecture holds true for each $(r,s)$-cabled knot over each $(p,q)$-torus knot if $r$ is not a number between 0 and $pqs$.
\end{thm}

\subsection{Plan of the paper} In Section \ref{cjp} we prove some properties of the colored Jones polynomial of cabled knots. In Section \ref{general} we prove a result about the relationship between the AJ conjecture and the strong AJ conjecture, and consequently prove one part of Theorem \ref{th2}. The proof of the other part of Theorem \ref{th2} is given in Sections \ref{q>2}--\ref{s=2}.

\subsection{Acknowledgment} We would like to thank X. Zhang for helpful discussions.

\section{The Colored Jones polynomial of cabled knots}

\label{cjp}

By \cite{Mo}, the formula for the colored Jones polynomial of the $(r,s)$-cabled knot $K^{(r,s)}$ over a knot $K$ is given by
\begin{equation}
\label{cjp-cable}
J_{K^{(r,s)}}(n)= t^{-rs(n^2-1)} \sum_{j=-\frac{n-1}{2}}^{\frac{n-1}{2}} t^{4rj(js+1)} J_K(2sj+1).
\end{equation}

\begin{lemma}
\label{J_C}
One has
\begin{eqnarray*}
&& J_{K^{(r,s)}}(n+2) - t^{-4rs(n+1)} J_{K^{(r,s)}}(n) \\
&=& t^{-2rs(n+1)} \left( t^{2r(n+1)}J_{K}(s(n+1)+1) - t^{-2r(n+1)}J_{K}(s(n+1)-1) \right).
\end{eqnarray*}
\end{lemma}

\begin{proof}
From Eq. \eqref{cjp-cable} we have
\begin{eqnarray*}
&& J_{K^{(r,s)}}(n+2) - t^{-4rs(n+1)} J_{K^{(r,s)}}(n) \\
&=& t^{-rs((n+2)^2-1)} \sum_{j=-\frac{n+1}{2}}^{\frac{n+1}{2}} t^{4rj(sj+1)} J_K(2sj+1)  - t^{-4rs(n+1)} t^{-rs(n^2-1)} \sum_{j=-\frac{n-1}{2}}^{\frac{n-1}{2}} t^{4rj(js+1)} J_K(2sj+1) \\
&=&t^{-rs((n+2)^2-1)}  \left( t^{r(n+1)(s(n+1)+2)}J_K(s(n+1)+1)+t^{r(n+1)(s(n+1)-2)} J_K(-s(n+1)+1) \right).
\end{eqnarray*}
The lemma follows since $J_K(-s(n+1)+1)=-J_{K}(s(n+1)-1)$.
\end{proof}

\begin{lemma}
\label{J_Cs_2}
One has
$$J_{K^{(r,2)}}(n+1) + t^{-2r(2n+1)}J_{K^{(r,2)}}(n) =  t^{-2rn}J_K(2n+1).$$
\end{lemma}

\begin{proof}
By Eq. \eqref{cjp-cable} we have
\begin{eqnarray*}
J_{K^{(r,2)}}(n+1)=t^{-2r((n+1)^2-1)} \sum_{k=-\frac{n}{2}}^{\frac{n}{2}} t^{4rk(2k+1)} J_K(4k+1).
\end{eqnarray*}
Set $k=-(j+\frac{1}{2}).$ Then
\begin{eqnarray*}
J_{K^{(r,2)}}(n+1) &=& t^{-2r((n+1)^2-1)} \sum_{j=\frac{n-1}{2}}^{-\frac{n+1}{2}} t^{4rj(2j+1)} J_K(-(4j+1))\\
      &=& t^{-2r((n+1)^2-1)} \Big( -\sum_{j=-\frac{n-1}{2}}^{\frac{n-1}{2}} t^{4rj(2j+1)} J_K(4j+1)+t^{2rn(n+1)} J_K(2n+1) \Big)\\
      &=&-t^{-2r(2n+1)} J_{K^{(r,2)}}(n)+t^{-2rn} J_K(2n+1).
\end{eqnarray*}
This proves Lemma \ref{J_Cs_2}.
\end{proof}

\begin{remark}
The proofs of Lemmas \ref{J_C} and \ref{J_Cs_2} are similar to those of Lemmas 1.1 and 1.5 in \cite{Tr}. See also Propositions 3 and  5 in \cite{Hi}, and Lemma 3.1 and Section 6.1 in \cite{RZ}.
\end{remark}

\section{On the strong AJ conjecture}

\label{general}

Let $\ve$ be the map reducing $t=-1$.

\begin{lemma}
\label{AJ}
Suppose the AJ conjecture holds true for a knot $K$ and the A-polynomial $A_K$ does not have any non-trivial $M$-factors. Then $\sqrt{\ve(\CA_K^{\sigma})} \subset \fp$.
\end{lemma}

\begin{proof}
Since the AJ conjecture holds true for $K$, we have $\ve(\alpha_K)=f(M)A_K$ for some non-zero $f(M) \in \BC(M)$. For every $\delta \in \CA_K$, by \cite[Lemma 2.5]{Tr}, there exist $g(t,M) \in \BC[t^{\pm 1},M]$ and $\gamma \in \CT$ such that $\delta=\frac{1}{g(t,M)}\,\gamma\,\alpha_K$ and $\varepsilon (g) \not= 0$. It follows that
$$
\varepsilon(\delta) = \frac{1}{\varepsilon (g(M))}\, \varepsilon(\gamma)\, \varepsilon(\alpha_K)=\frac{1}{\varepsilon (g(M))}\, \varepsilon(\gamma) \, f(M)A_K.$$

Let $h=\frac{f(M)}{\varepsilon (g(M))}\, \varepsilon(\gamma) \in \BC(M)[L^{\pm 1}]$. Then $\ve(\delta)=h A_K$. Since $A_K$ does not have any non-trivial $M$-factors, we must have $h \in \BC[M^{\pm 1}, L^{\pm 1}]=\ft$. Hence $\varepsilon(\delta) \in A_K \cdot \ft$, the ideal of $\ft$ generated by $A_K$. It follows that $\varepsilon(\CA_K) \subset A_K \cdot \ft$ and thus  $\varepsilon(\CA_K^{\sigma}) \subset (A_K \cdot \ft)^{\sigma}=\fp$. Hence $\sqrt{\varepsilon(\CA_K^{\sigma})} \subset \sqrt{\fp}=\fp$. 
\end{proof}

\begin{lemma}
\label{AJ0}
Suppose there exists $P \in \CA_K^{\sigma}$ such that $\ve(P)=M^kL^l(A_K)^{2m}$ for some integers $k,l,m$. Then $\fp \subset \sqrt{\ve(\CA_K^{\sigma})}$.
\end{lemma}

\begin{proof}
It is known that $A_K(M^{\pm 1},L^{\pm 1})=\eta M^aL^b A_K(M,L)$ where $a,b \in \BZ$ and $\eta =\pm 1$, see \cite{CCGLS}. Since $\sigma(P)=P$, we have
$$M^kL^l(A_K)^{2m}=M^{-k}L^{-l}(\sigma(A_K))^{2m}=M^{-k}L^{-l}(\eta M^aL^b A_K)^{2m}.$$
It follows that $k=am$ and $l=bm$. Hence $\ve(P)=M^{am}L^{bm}(A_K)^{2m}$.

If $u \in \fp$ then $u=vA_K$ for some $v \in \BC[M^{\pm 1}, L^{\pm 1}]$. Since $u=\sigma(u)$, we have $vA_k=\sigma(v) \sigma(A_K)=\sigma(v) \eta M^aL^b A_K$ which means that $v=\eta M^aL^b\sigma(v)$. Hence $M^{-a}L^{-b}v^2=\sigma (M^{-a}L^{-b}v^2).$ Let $w=M^{-a}L^{-b}v^2$. Then $\sigma(w)=w$.
We have 
$$u^{2m}=v^{2m}(A_K)^{2m}=\ve((M^{-a}L^{-b}v^2)^mP)=\ve(w^mP) \in \ve(\CA_K^{\sigma}),$$
which implies that $u \in \sqrt{\ve(\CA_K^{\sigma})}$. This proves Lemma \ref{AJ0}.
\end{proof}

Let $T=T(p,q)$ be the $(p,q)$-torus knot and $C=T^{(r,s)}$ be the $(r,s)$-cabled knot over $T$. By \cite{RZ}, the formula for the A-polynomial of $C$ is given by 
$$A_C(M,L)=
\begin{cases} (L-1)(L^2-M^{-2pqs^2})(L^2-M^{-2rs}) &\mbox{if $s>2$ is odd and $q>2$}, \\ 
(L-1)(L+M^{-2ps^2})(L^2-M^{-2rs}) & \mbox{if $s>2$ is odd and $q>2$}, \\ 
(L-1)(L-M^{-pqs^2})(L^2-M^{-2rs}) & \mbox{if $s>2$ is even}, \\ 
(L-1)(L-M^{-4pq})(L+M^{-2r}) & \mbox{if $s=2$}.
\end{cases}$$

Suppose $r$ is not a number between 0 and $pqs$. Ruppe and Zhang \cite{RZ} have confirmed the AJ conjecture for the cabled knot $C$. Since the A-polynomial of $C$ does not have any non-trivial $M$-factors, Lemma \ref{AJ} implies that $\sqrt{\ve(\CA_C^{\sigma})} \subset \fp$. Hence to prove the strong AJ conjecture for $C$, we only need to show that 
\begin{equation}
\label{done}
\fp \subset \sqrt{\ve(\CA_C^{\sigma})}.
\end{equation} 

To prove \eqref{done} we will apply Lemma \ref{AJ0}. As in \cite{RZ}, the proof of \eqref{done} is divided into the following 4 cases:

(1) $s$ is odd and $q>2$ (Section \ref{q>2}),

(2) $s$ is odd and $q=2$ (Section \ref{q=2}),

(3) $s>2$ is even (Section \ref{s>2}),

(4) $s=2$ (Section \ref{s=2}).\\
Moreover, in the case $s>2$ we let $G_1(n)=t^{2r(n+1)}J_{T}(s(n+1)+1)$ and $G_2(n)=t^{-2r(n+1)}J_{T}(s(n+1)-1).$ Then, by Lemma \ref{J_C}, we have
\begin{equation} 
\label{f_1-f_2}
t^{2rs}M^{rs} (L^2-t^{-4rs}M^{-2rs})J_{C}=G_1-G_2.
\end{equation}

The following lemma will be useful in the proof of \eqref{done}.

\begin{lemma}
\label{tM}
Suppose $h_1(n), \cdots, h_k(n) \in \CR[t^{\pm 2n}]$. There exists $P \in \CR[L^{\pm 1}]$ such that $\sigma(P)=P$, $\ve(P)=(L+L^{-1}-2)^m$ for some $m \ge 1$, and $Ph_i=0$ for all $i=1,\cdots, k$.
\end{lemma}

\begin{proof}

Suppose $h(n) \in \CR[t^{\pm 2n}]$. Write $h(n) = \sum_j \lambda_{j} \, t^{2k_{j}n}$, where  $\lambda_{j} \in \CR$ and $k_{j} \in \BZ$. Since $t^{2k_{j}n}$ is annihilated by $L+L^{-1}-t^{2k_{j}}-t^{-2k_{j}}$, $h(n)$ is annihilated by $$P_h:=\prod_j (L+L^{-1}-t^{2k_{j}}-t^{-2k_{j}}).$$

Let $P=P_{h_1} \cdots P_{h_n}$. Then $P$ satisfies the conditions of Lemma \ref{tM}.
\end{proof}

\section{Case $s>2$ is odd and $q>2$}

\label{q>2}

From \cite[Lemma 1.1]{Tr} we have $J_T(n+2)-t^{-4pq(n+1)}J_T(n) \in \CR[t^{\pm 2n}]$, i.e.
\begin{equation}
\label{lem1.1}
(L^{2}-t^{-4pq}M^{-2pq})J_T \in \CR[M^{\pm 1}].
\end{equation}

\begin{lemma} 
\label{J_T}
For all positive integers $m$, one has
$$(L^{2m}-t^{-4pqm^2}M^{-2pqm})J_T \in \CR[M^{\pm 1}].$$
\end{lemma}

\begin{proof}
From \eqref{lem1.1} we have $(L^{2}M^{2pq}-t^{4pq})J_T \in \CR[M^{\pm 1}].$ It follows that $$((L^{2}M^{2pq})^m-t^{4pqm})J_T \in \CR[ M^{\pm 1}].$$ Since $(L^{2}M^{2pq})^m=t^{-4pqm(m-1)}L^{2m}M^{2pqm}$, we have $(L^{2m}M^{2pqm}-t^{4pqm^2})J_T \in \CR[M^{\pm 1}].$ The lemma  follows since $L^{2m}M^{2pqm}=t^{8pqm^2}M^{2pqm}L^{2m}$.
\end{proof}

\begin{proposition}
\label{P_1P_2}
One has 
\begin{eqnarray*}
(L^2M^{2pqs^2}+L^{-2}M^{-2pqs^2}-t^{4r-4pqs}-t^{-4r+4pqs+8pqs^2}) G_1 &\in& \CR[ M^{\pm 1}],\\
(L^2M^{2pqs^2}+L^{-2}M^{-2pqs^2}-t^{-4r+4pqs}-t^{4r-4pqs+8pqs^2}) G_2 &\in& \CR[M^{\pm 1}].
\end{eqnarray*}
\end{proposition}

\begin{proof} From Lemma \ref{J_T} we have 
$$J_{T}(s(n+3)+1)-t^{-4pqs(sn+2s+1)}J_{T}(s(n+1)+1) \in \CR[t^{\pm 2n}].$$
It follows that $G_1(n+2)-t^{4r-4pqs(sn+2s+1)}G_1(n) \in \CR[t^{\pm 2n}]$. Hence $$(L^2-t^{4r-4pqs(2s+1)}M^{-2pqs^2})G_1 \in \CR[M^{\pm 1}],$$
which implies that $(L^2M^{2pqs^2}-t^{4r-4pqs})G_1 \in \CR[M^{\pm 1}]$. Hence
$$(1-t^{-4r+4pqs}L^{-2}M^{-2pqs^2})(L^2M^{2pqs^2}-t^{4r-4pqs}) G_1 \in \CR[M^{\pm 1}].$$
The proof for $G_2$ is similar.
\end{proof}

Let 
\begin{eqnarray*}
P_1 &:=& L^2M^{2pqs^2}+L^{-2}M^{-2pqs^2}-t^{4r-4pqs}-t^{-4r+4pqs+8pqs^2},\\
P_2 &:=& L^2M^{2pqs^2}+L^{-2}M^{-2pqs^2}-t^{-4r+4pqs}-t^{4r-4pqs+8pqs^2}.
\end{eqnarray*}
Note that $\sigma(P_i)=P_i$ and $P_1P_2=P_2P_1$. By Proposition \ref{P_1P_2}, we have $P_iG_i \in \CR[M^{\pm 1}]$ for $i=1,2$. Let $P=P_1P_2$. Then $\sigma(P)=P$ and $PG_i \in \CR[M^{\pm 1}]$.

By Lemma \ref{tM}, we can choose $Q \in \CR[L^{\pm 1}]$ such that $\sigma(Q)=Q$, $\ve(Q)=(L+L^{-1}-2)^m$ for some $m \ge 1$, and $Q(PG_i)=0$ for all $i=1,2$. 

From \eqref{f_1-f_2} we have $QPM^{rs} (L^2-t^{-4rs}M^{-2rs})J_{C}=0$. Let
$$R := M^{rs}QPM^{rs} (L^2-t^{-4rs}M^{-2rs})+M^{-rs}QPM^{-rs} (L^{-2}-t^{-4rs}M^{2rs}).$$
Then $\sigma(R)=R$. Since $\CA_C$ is invariant under $\sigma$ (by \cite{Ga08}), we have $RJ_{C}=0$. Note that
\begin{eqnarray*}
\ve(R) &=& (L^2M^{2rs}+L^{-2}M^{-2rs}-2) \ve(QP)\\
       &=& (L^2M^{2rs}+L^{-2}M^{-2rs}-2)(L^2M^{2pqs^2}+L^{-2}M^{-2pqs^2}-2)^2(L+L^{-1}-2)^m.
\end{eqnarray*}
Let $k:=\max\{m,2\}$ and $S:=(L^2M^{2rs}+L^{-2}M^{-2rs}-2)^{k-1}(L^2M^{2pqs^2}+L^{-2}M^{-2pqs^2}-2)^{k-2}(L+L^{-1}-2)^{k-m}.$ We have $\sigma(S)=S$ and 
\begin{eqnarray*}
\ve(SR) &=& \left( (L^2M^{2rs}+L^{-2}M^{-2rs}-2)(L^2M^{2pqs^2}+L^{-2}M^{-2pqs^2}-2)(L+L^{-1}-2) \right)^k \\
        &=& M^{2(r+pqs)sk} L^{-5k} \left( (L-1)(L^2-M^{-2pqs^2})(L^2-M^{-2rs}) \right)^{2k}.
\end{eqnarray*}

Since $SR \in \CA^{\sigma}_C$ and $\ve(SR)=M^{2(r+pqs)sk} L^{-5k} (A_C)^{2k}$, Lemma \ref{AJ0} implies that $\fp \subset \sqrt{\ve(\CA_C^{\sigma})}$. This proves Theorem \ref{th2} in the case $s>2$ is odd and $q>2$.

\section{Case $s>2$ is odd and $q=2$}

\label{q=2}

From \cite[Lemma 1.5]{Tr} we have $J_T(n+1)+t^{-2p(2n+1)}J_T(n) \in \CR[t^{\pm 2n}]$, i.e.
\begin{equation}
\label{lem1.2}
(L+t^{-2p}M^{-2p})J_T \in \CR[M^{\pm 1}].
\end{equation}

\begin{lemma} 
\label{J_Tq_2}
When $q=2$, for all positive integers $m$, one has
$$(L^m-(-1)^m t^{-2pm^2}M^{-2pm})J_{T} \in \CR[M^{\pm 1}].$$
\end{lemma}

\begin{proof}
From \eqref{lem1.2} we have $(LM^{2p}+t^{2p})J_T \in \CR[M^{\pm 1}].$ It follows that $$((LM^{2p})^m-(-t^{2p})^m)J_T \in \CR[ M^{\pm 1}].$$ Since $(LM^{2p})^m=t^{-2pm(m-1)}L^{m}M^{2pm}$, we have $(L^mM^{2pm}-(-1)^m t^{2pm^2})J_T \in \CR[M^{\pm 1}].$ The lemma follows since $L^mM^{2pm} = t^{4pm^2}M^{2pm}L^m$.
\end{proof}

\begin{proposition}
\label{P_1P_2q_2}
When $q=2$ and $s$ is odd, one has 
\begin{eqnarray*}
(LM^{2ps^2}+L^{-1}M^{-2ps^2}+t^{2r-2ps(s+2)}+t^{-2r+2ps(3s+2)}) G_1 &\in& \CR[ M^{\pm 1}],\\
(LM^{2ps^2}+L^{-1}M^{-2ps^2}+t^{-2r-2ps(s-2)}+t^{2r+2ps(3s-2)}) G_2 &\in& \CR[ M^{\pm 1}].
\end{eqnarray*}
\end{proposition}

\begin{proof} From Lemma \ref{J_Tq_2} we have 
$$J_{T}(s(n+2)+1)+t^{-2ps(2sn+3s+2)}J_{T}(s(n+1)+1) \in \CR[t^{\pm 2n}].$$
It follows that $G_1(n+1)+t^{2r-2ps(2sn+3s+2)}G_1(n) \in \CR[t^{\pm 2n}]$. Hence $$(L+t^{2r-2ps(3s+2)}M^{-2ps^2})G_1 \in \CR[M^{\pm 1}],$$
which implies that $(LM^{2ps^2}+t^{2r-2ps(s+2)})G_1 \in \CR[M^{\pm 1}]$. Hence
$$(1+t^{-2r+2ps(s+2)}L^{-1}M^{-2ps^2})(LM^{2ps^2}+t^{2r-2ps(s+2)}) G_1 \in \CR[M^{\pm 1}].$$
The proof for $G_2$ is similar.
\end{proof}

By applying Proposition \ref{P_1P_2q_2}, we can show that $\fp \subset \sqrt{\ve(\CA_C^{\sigma})}$ as in Section \ref{q>2}.

\section{Case $s>2$ is even}

\label{s>2}

\begin{proposition}
\label{prop.11}
When $s$ is even, one has 
\begin{eqnarray*}
(LM^{pqs^2}+L^{-1}M^{-pqs^2}-t^{2r-pqs(s+2)}-t^{-2r+pqs(3s+2)}) G_1 &\in& \CR[ M^{\pm 1}],\\
(LM^{pqs^2}+L^{-1}M^{-pqs^2}-t^{-2r-pqs(s-2)}-t^{2r+pqs(3s-2)}) G_2 &\in& \CR[ M^{\pm 1}].
\end{eqnarray*}
\end{proposition}

\begin{proof} From Lemma \ref{J_T} we have 
$$J_{T}(s(n+2)+1)-t^{-pqs(2sn+3s+2)}J_{T}(s(n+1)+1) \in \CR[t^{\pm 2n}].$$
It follows that $G_1(n+1)-t^{2r-pqs(2sn+3s+2)}G_1(n) \in \CR[t^{\pm 2n}]$. Hence $$(L-t^{2r-pqs(3s+2)}M^{-pqs^2})G_1 \in \CR[M^{\pm 1}],$$
which implies that $(LM^{pqs^2}-t^{2r-pqs(s+2)})G_1 \in \CR[M^{\pm 1}]$. Hence
$$(1-t^{-2r+2pq(s+2)}L^{-1}M^{-2ps^2})(LM^{pqs^2}-t^{2r-pqs(s+2)}) G_1 \in \CR[M^{\pm 1}].$$
The proof for $G_2$ is similar.
\end{proof}

By applying Proposition \ref{prop.11}, we can show that $\fp \subset \sqrt{\ve(\CA_C^{\sigma})}$ as in Section \ref{q>2}.

\section{Case $s=2$}

\label{s=2}

In the case $s=2$ we let $G(n)=J_K(2n+1)$. Then, by Lemma \ref{J_Cs_2}, we have 
\begin{equation}
\label{f}
M^r(L+t^{-2r}M^{-2r})J_{C} = G.
\end{equation}

\begin{proposition}
\label{ff}
One has
$$(LM^{4pq}+L^{-1}M^{-4pq}-1-t^{8pq})G \in  \CR[M^{\pm 1}].$$
\end{proposition}

\begin{proof} From \cite[Lemma 1.1]{Tr} we have 
$$J_{T}(2n+3)-t^{-8pq(n+1)}J_{T}(2n+1) \in \CR[t^{\pm 2n}],$$
i.e. $(L-t^{-8pq}M^{-4pq})G \in \CR[M^{\pm 1}].$ It follows that $(LM^{4pq}-1)G \in \CR[M^{\pm 1}]$. Hence
$$(1-L^{-1}M^{-4pq})(LM^{4pq}-1)G \in \CR[M^{\pm 1}].$$
The proposition follows.
\end{proof}

Let $P=LM^{4pq}+L^{-1}M^{-4pq}-1-t^{8pq}$. Note that $\sigma(P)=P$. By Proposition \ref{ff} we have $PG \in  \CR[ M^{\pm 1}]$. By Lemma \ref{tM}, we can choose $Q \in \CR[L^{\pm 1}]$ such that $\sigma(Q)=Q$, $\ve(Q)=(L+L^{-1}-2)^m$ for some $m \ge 1$, and $Q(PG)=0$. Eq. \eqref{f} then implies that $QPM^r(L+t^{-2r}M^{-2r})J_{C}=0.$ Let
$$R := M^{r}QPM^r(L+t^{-2r}M^{-2r})+M^{-r}QPM^{-r}(L^{-1}+t^{-2r}M^{2r}).$$
Then $\sigma(R)=R$. Since $\CA_C$ is invariant under $\sigma$, we have $RJ_{C}=0$. Note that
\begin{eqnarray*}
\ve(R) &=& (LM^{2r}+L^{-1}M^{-2r}-2)(LM^{4pq}+L^{-1}M^{-4pq}-2)(L+L^{-1}-2)^m.
\end{eqnarray*}
Let $S:=(LM^{2r}+L^{-1}M^{-2r}-2)^{m-1}(LM^{4pq}+L^{-1}M^{-4pq}-2)^{m-1}.$ We have $\sigma(S)=S$ and 
\begin{eqnarray*}
\ve(SR) &=& \left( (LM^{2r}+L^{-1}M^{-2r}+2)(LM^{4pq}+L^{-1}M^{-4pq}-2)(L+L^{-1}-2) \right)^m \\
        &=& M^{2(r+2pq)m} L^{-3m} \left( (L-1)(L+M^{-2r})(L-M^{-4pq}) \right)^{2m}.
\end{eqnarray*}

Since $SR \in \CA^{\sigma}_C$ and $\ve(SR)=M^{2(r+2pq)m} L^{-3m} (A_C)^{2m}$, Lemma \ref{AJ0} implies that $\fp \subset \sqrt{\ve(\CA_C^{\sigma})}$. This proves Theorem \ref{th2} in the case $s=2$.

\end{document}